\documentclass[leqno,12pt]{article} 
\usepackage{amsmath, amssymb}
\usepackage{amsthm} 
\theoremstyle{plain} 
\newtheorem{theorem}{\indent\sc Theorem}[section]
\newtheorem{lemma}[theorem]{\indent\sc Lemma}
\newtheorem{corollary}[theorem]{\indent\sc Corollary}
\newtheorem{proposition}[theorem]{\indent\sc Proposition}

\theoremstyle{definition} 
\newtheorem{definition}[theorem]{\indent\sc Definition}
\newtheorem{remark}[theorem]{\indent\sc Remark}
\newtheorem{example}[theorem]{\indent\sc Example}

\makeatletter
\def\address#1#2{\begingroup
\noindent\parbox[t]{7.8cm}{%
\small{\scshape\ignorespaces#1}\par\vskip1ex
\noindent\small{\itshape E-mail address}%
\/: #2\par\vskip4ex}\hfill%
\endgroup}%
\makeatother
\title{\uppercase{Biharmonic homomorphisms between Riemannian Lie groups}} 
\author{
\bigskip \\
\textsc{Mohamed Boucetta$^{*}$ and Seddik Ouakkas} 
}
\date{} 
\newcommand{\prs}{\langle\;,\;\rangle}
\newcommand{\too}{\longrightarrow}
\newcommand{\om}{\omega}
\newcommand{\Ad}{{\mathrm{Ad}}}
\newcommand{\esp}{\quad\mbox{and}\quad}

\newcommand{\G}{{\mathfrak{g}}}

\newcommand{\h}{{\mathfrak{h} }}
\newcommand{\n}{{\mathfrak{n} }}
\newcommand{\ad}{{\mathrm{ad}}}
\newcommand{\tr}{{\mathrm{tr}}}
\newcommand{\ric}{{\mathrm{ric}}}

\newcommand{\spa}{{\mathrm{span}}}

\newcommand{\na}{\nabla}
\newcommand{\wi}{\widetilde}
\newcommand{\al}{\alpha}

\newcommand{\Ga}{\Gamma}

\newcommand{\la}{\lambda}

\newcommand{\de}{\delta}

\font\bb=msbm10

\def\R{\hbox{\bb R}}

\begin{document}

\maketitle

\footnote{ 
}

\begin{abstract}
A Lie group $G$ endowed with a left invariant Riemannian metric $g$ is called \emph{Riemannian Lie group}. Harmonic and biharmonic maps between Riemannian manifolds is an important area of investigation. In this paper, we study different aspects of harmonic and biharmonic homomorphisms between Riemannian Lie groups. We show that this class of biharmonic maps can be  used at the first level  to build examples but, as we will see through this paper, its study will lead to some interesting mathematical problems in the theory of Riemannian Lie groups.
\end{abstract}

\section{Introduction}\label{section1}

Let $\phi:(M,g)\too(N,h)$ be a smooth map between two Riemannian manifolds with $m=\dim M$ and $n=\dim N$. We  denote by $\na^M$ and $\na^N$ the Levi-Civita connexions associated respectively to $g$ and $h$ and  by $T^\phi N$ the vector bundle over $M$ pull-back of $TN$ by $\phi$. It is  an Euclidean vector bundle and the tangent map of $\phi$ is a bundle homomorphism $d\phi:TM\too T^\phi N$. Moreover, $T^\phi N$  carries a connexion $\na^\phi$ pull-back of $\na^N$ by $\phi$ and there is  a connexion on the vector bundle  $\mathrm{End}(TM,T^\phi N)$ given by
\[ (\na_X A)(Y)=\na_X^\phi A(Y)-A\left(\na_X^MY \right),\quad X,Y\in\Ga(TM), A\in \Ga\left(\mathrm{End}(TM,T^\phi N)\right).\]
The map $\phi$ is called harmonic if it is a critical point of the energy $E(\phi)=\frac12\int_M|d\phi|^2\nu_g$. The corresponding Euler-Lagrange equation for the energy is given by the vanishing of the tension field  
\begin{equation*}\label{eqtension} \tau(\phi)=\tr_{g}\na d\phi=\sum_{i=1}^m(\na_{E_i}d\phi)(E_i),     \end{equation*}
where $(E_i)_{i=1}^m$ is a local frame of orthonormal vector fields. Note that $\tau(\phi)\in\Ga(T^\phi N)$.
The map $\phi$ is called biharmonic if it is a critical point of the
 bienergy of $\phi$  defined by $E_2(\phi)=\frac12\int_M|\tau(\phi)|^2\nu_g$. The corresponding Euler-Lagrange equation for the bienergy is given by the vanishing of the bitension field 
 \begin{eqnarray*}\label{eqbitension} \tau_2(\phi)&=&-\tr_{g}(\na^\phi)^2_{.\;,\;.}\tau(\phi)-
 \tr_{g}R^N(\tau(\phi),d\phi(\;.\;))d\phi(\;{\bf.}\;)\nonumber\\&=&-\sum_{i=1}^m\left((\na^\phi)^2_{E_i,E_i}\tau(\phi)   +R^N(\tau(\phi),d\phi(E_i))d\phi(E_i)\right),
      \end{eqnarray*}
where $(E_i)_{i=1}^m$ is a local frame of orthonormal vector fields, $(\na^\phi)^2_{X,Y}=\na^\phi_X\na^\phi_Y-\na^\phi_{\na_X^MY}$ and $R^N$ is the curvature of $\na^N$ given by
\[R^N(X,Y)=\na_X^N\na_Y^N-\na_Y^N\na_X^N-\na_{[X,Y]}^N. \]
The theory of biharmonic maps is old and rich and has gained a growing interest in the last decade see \cite{baird,sario} and others. The theory of biharmonic maps into Lie groups, symmetric spaces or
homogeneous spaces has been extensively studied related to  
integrable systems  (see for instance \cite{dai, uhlenbeck, wood}). In particular, harmonic maps of Riemann surfaces into
compact Lie groups equipped with a bi-invariant Riemannian metric are called principal
chiral models and intensively studied as toy models of gauge theory in mathematical physics \cite{zakr}.
 Curiously, there is no detailed study of harmonic or biharmonic homomorphisms between Riemannian Lie groups\footnote{A biharmonic homomorphism between Riemannian Lie groups is a homomorphism of Lie groups $\phi:G\too H$ which is also biharmonic where $G$ and $H$ are endowed with left invariant Riemannian metrics.}.  To our knowledge, the only works on this topic are \cite{park, park2} where the authors studied harmonic inner automorphisms of a compact semi-simple Lie group endowed with a left invariant Riemannian metric. In this paper, 
 we investigate   biharmonic homomorphisms between Riemannian Lie groups. At first sight, this class can be mainly used  to build examples but, as we will see through this paper, its study  gives rise to some interesting mathematical problems in the theory of Riemannian Lie groups. Moreover, this class can be enlarged non trivially without extra work as follows. Let $\phi: G\too H$ be a biharmonic homomorphism between two Riemannian Lie groups,  $\Ga_1$ and $\Ga_2$ are two discrete subgroups of $G$ and $H$, respectively, with $\phi(\Ga_1)\subset\Ga_2$. Then $\Ga_1/G$ and $\Ga_2/H$ carry two  Riemannian metrics and $\pi\circ\phi:G\too \Ga_2/H$ factor to a smooth map $\wi\phi:\Ga_1/G\too \Ga_2/H$ which  is biharmonic (harmonic if $\phi$ is harmonic).\\
 The paper is organized as follows. In Section \ref{section2}, we characterize at the level of Lie algebras harmonic and biharmonic homomorphisms between Riemannian Lie groups and we give some of their general properties. We obtain some  results which even they are immediate  are interesting (see Propositions \ref{prbi}, \ref{pr4}, \ref{pr4bis} and Example \ref{exem}).
 In Section \ref{section3}, we study harmonic automorphisms of a Riemannian Lie group. We show that on a compact Riemannian Lie group there is a finite number of harmonic inner automorphisms, up to inner isometries (see Theorem \ref{finite}).
 We show also that a left invariant Riemannian metric on a Lie group satisfies the property that any inner automorphism is harmonic iff it is bi-invariant (see Theorem \ref{park}). These two results have been obtained by Park in \cite{park} in the particular case of $\mathrm{SU}(2)$. On the other hand,
 we introduce the notion of the harmonic cone associated to a left invariant Riemannian metric and we show that it is a Riemannian invariant of the metric in some sense. We compute the dimension of this cone when the Lie group is unimodular (see Proposition \ref{compute}). In Sections \ref{section4}, we give an useful description of submersions between Riemannian Lie groups and their tension fields (see Propositions \ref{submersion} and \ref{pr6}). Based on this, we prove in Section \ref{section5} two important results on biharmonic Riemannian submersions between Riemannian Lie groups (see Theorems \ref{theo1} and \ref{theo2}).
  In Section \ref{section6}, we give many situations where harmonicity and biharmonicity of a homomorphism between two Riemannian Lie groups are equivalent. Some situations are particular cases of known results but many  are new and specific to our context (see Theorems \ref{theo3} and \ref{theo5}).
 In Section \ref{section7}, based on the results of Sections \ref{section4}-\ref{section5}, we give some general methods to build large classes of harmonic or biharmonic homomorphisms between Riemannian Lie groups. In Section \ref{section8}, we clarify completely the situation in dimension 2. 
 
 {\bf Notations and conventions.} Trough this paper all considered Lie groups  are supposed to be connected. A Lie group $G$ endowed with a left invariant Riemannian metric $g$ is called \emph{Riemannian Lie group}. Its Lie algebra $\G=T_eG$ endowed with the scalar product $g(e)$ is called \emph{Euclidean Lie algebra}. For any $u\in\G$, we denote by $\ad_u:\G\too\G$ the linear map given by $\ad_uv=[u,v]$. The group $G$ is \emph{unimodular} iff, for any $u\in\G$, $\tr(\ad_u)=0$. A Riemannian metric $g$ on  $G$ is bi-invariant  if it is left an right invariant. This is equivalent to $\ad_u$ is skew-symmetric with respect $g(e)$ for any $u\in\G$. For any $a\in G$, $L_a$ and $R_a$ denote, respectively, the left multiplication and the right multiplication by $a$ and $i_a=L_a\circ R_{a^{-1}}$ is the inner automorphism associated to $a$. We denote by $\Ad_a$ the differential of $i_a$ at $e$. Through this paper maps between Lie groups are  homomorphisms. In particular,
a submersion (resp. Riemannian submersion) between Riemannian Lie groups is an homomorphism of Lie groups which is a submersion (resp. Riemannian submersion). \\
Finally, if $F:(V,\prs_1)\too(W,\prs_2)$ is a linear map between two Euclidean vector space, we denote by $F^*:W\too V$ the adjoint given by the relation $\langle F(w),v\rangle_1=\langle w,F(v)\rangle_2$. 
 
\section{Biharmonic homomorphisms between Riemannian Lie groups: general properties and first examples}
\label{section2}

For an homomorphism between two Riemannian Lie group the tension field and the bitension field are left invariant and the harmonicity and the biharmonicity are equivalent to the vanishing of two vectors in the Lie algebra of the target group. We express  these two vectors in two different ways and we deduce some immediate properties. We investigate the cases holomorphic homomorphisms between K\"ahlerian Lie groups and Riemannian immersions.

\subsection{Tension field and bitension field for homomorphisms between Riemannian Lie groups}

Let $(G,g)$ be a Riemannian Lie group. If $\G=T_eG$ is its Lie algebra and $\prs_\G=g(e)$ then there exists a unique bilinear map $A:\G\times\G\too\G$ called the Levi-Civita product associated to $(\G,\prs_\G)$ given by the formula:
\begin{equation}\label{lc}
 2\langle A_uv,w\rangle_\G=\langle[u,v]^\G ,w\rangle_\G+\langle[w,u]^\G
,v\rangle_\G+\langle[w,v]^\G ,u\rangle_\G.
\end{equation} $A$ is entirely determined by the following properties:
\begin{enumerate}\item for any $u,v\in\G$, $A_uv-A_vu=[u,v]^\G$,
\item for any $u,v,w\in\G$, $\langle A_uv,w\rangle_\G+\langle v,A_uw\rangle_\G=0$.
\end{enumerate}
If we denote by $u^\ell$ the left invariant vector field on $G$ associated to $u\in\G$ then the Levi-Civita connection associated to $(G,g)$ satisfies $\na_{u^\ell}v^\ell=\left(A_uv \right)^\ell$. The couple $(\G,\prs_\G)$ defines a vector say $U^\G\in\G$ by
\begin{equation*}\label{eq1}
  \langle U^\G,v\rangle_\G=\tr(\ad_v),\quad\mbox{for any}\; v\in\G.
\end{equation*} One can deduce  easily from \eqref{lc} that, for any orthonormal basis $(e_i)_{i=1}^n$ of $\G$, 
\begin{equation}\label{ug} U^\G=\sum_{i=1}^nA_{e_i}e_i. \end{equation}
Note that $\G$ is  unimodular iff $U^\G=0$. We denote by $\mathrm{Inisom}(g)$ the subgroup of $G$ consisting of $a\in G$ such that the inner automorphism $i_a$ is an isometry. We have $$\mathrm{Inisom}(g)=\left\{a\in G,\Ad_a^*\circ\Ad_a=\mathrm{id}_\G \right\}.$$ Thus $\mathrm{Inisom}(g)$ is a closed subgroup containing the center $Z(G)$.We denote by $\mathrm{Kill}(g)$ its Lie algebra  given by $$\mathrm{Kill}(g)=\{u\in\G,\ad_u+\ad_u^*=0\}.$$ Remark that $\mathrm{Kill}(g)$ can be identified with the Lie algebra of left invariant Killing vector fields of $g$ and if $G$ is nilpotent then $\mathrm{Inisom}(g)=Z(G)$.

Let $\phi:(G,g)\too (H,h)$ be a Lie groups homomorphism  between two Riemannian Lie groups.  The differential $\xi:\G\too\h$ of $\phi$ at $e$ is a Lie algebra homomorphism. Define $U^\xi\in\h$ by
\begin{equation}\label{uxi}U^\xi=\sum_{i=1}^nB_{\xi(e_i)}{\xi(e_i)},\end{equation} where $(e_i)_{i=1}^n$ is an orthonormal basis of $\G$ and $B$ is the Levi-Civita product associated to $(\h,\prs_\h)$.
There is a left  action of $G$ on $\Ga(T^\phi H)$ given by 
\[ (a.X)(b)= T_{\phi(ab)}L_{\phi(a^{-1})} X(ab),\quad a,b\in G, X\in \Ga(T^\phi H).   \]
A section $X$ of $T^\phi H$ is called left invariant if, for any $a\in G$, $a.X=X$. For any left invariant section $X$ of $T^\phi H$, we have for any $a\in G$,
$ X(a)=(X(e))^\ell(\phi(a)).$ 
Thus  the space of left invariant sections is isomorphic to the Lie algebra $\h$.
Since $\phi$ is a homomorphism of Lie groups and $g$ and $h$ are left invariant, one can see easily that $\tau(\phi)$ and $\tau_2(\phi)$ are left invariant and hence $\phi$ is harmonic (resp. biharmonic) iff $\tau(\phi)(e)=0$ (resp. $\tau_2(\phi)(e)=0$).  
 So we get the following proposition.
\begin{proposition}\label{pr1}Let $\phi:G\too H$ be an homomorphism between two Riemannian Lie groups. Then $\phi$ is harmonic $($resp. biharmonic$)$ iff
\begin{eqnarray*}\label{tension}\tau(\xi):&=&\tau(\phi)(e)=U^\xi-\xi(U^\G)=0,\\ \tau_2(\xi):&=&\tau_2(\phi)(e)=-\sum_{i=1}^n\left(B_{\xi(e_i)}B_{\xi(e_i)}\tau(\xi)
+K^H(\tau(\xi),\xi(e_i))\xi(e_i)\right)+B_{\xi(U^\G)}\tau(\xi)=0,\nonumber
\end{eqnarray*}where $\xi:\G\too\h$ is the differential of $\phi$ at $e$, $B$ is the Levi-Civita product associated to $(\h,\prs_\h)$
 and $K^H$ is the curvature of $B$ given by $K^H(u,v)=[B_u,B_v]-B_{[u,v]}$.

\end{proposition}

\begin{remark}\label{rem1}\begin{enumerate}\item It is obvious from the definition of $\tau(\xi)$ that if $\xi$ preserves the Levi-Civita products then $\phi$ is harmonic. 

\end{enumerate}

\end{remark}
The following proposition follows easily from the definition of $\tau(\xi)$. It is a particular case of Proposition \ref{prbi}.
\begin{proposition}\label{prab} Let $\phi:(G,g)\too (H,h)$ be an  homomorphism  between two Riemannian Lie groups where $H$ is abelian. Then $\phi$ is biharmonic and it is harmonic when $\G$ is unimodular. In particular, any character $\chi:G\too\R$ is biharmonic and it is harmonic when $\G$ is unimodular.

\end{proposition}

We give now another expression of $\tau(\xi)$ and $\tau_2(\xi)$ which will be useful later. 
\begin{proposition}\label{pr2}With the notations above, we have for any $u\in\h$,
\begin{eqnarray*}
\langle U^\xi,u\rangle_\h&=&\tr(\xi^*\circ\ad_u\circ \xi),\\
\langle\tau_2(\xi),u\rangle_\h&=&\tr(\xi^*\circ(\ad_u+\ad_u^*) \circ\ad_{\tau(\xi)} \circ\xi)
-\langle [u,\tau(\xi)]^\h,\tau(\xi)\rangle_\mathfrak{h}-\langle
[{\tau(\xi)},U^\xi]^\h,u\rangle_\mathfrak{h}.
\end{eqnarray*}In particular, $U^\xi\in (\mathrm{Kill}(h))^\perp$.

\end{proposition}
\begin{proof} Let $(e_i)_{i=1}^n$ be an orthonormal basis of $\G$.
For any $u\in\h$, we have
\begin{eqnarray*}
\langle U^\xi,u\rangle_\h&=&\sum_{i=1}^n\langle B_{\xi(e_i)}\xi(e_i),u\rangle_\h
\stackrel{\eqref{lc}}=\sum_{i=1}^n\langle [u,{\xi(e_i)}]^\h,\xi(e_i)\rangle_\h\\
&=&\sum_{i=1}^n\langle\xi^*\circ\ad_u \circ{\xi(e_i)},e_i\rangle_\G
=\tr(\xi^*\circ\ad_u\circ \xi),
\end{eqnarray*}which gives the first relation. From this relation, one can deduce that if $\ad_u^*=-\ad_u$ then $\langle U^\xi,u\rangle_\h=0$ and hence $U^\xi\in(\mathrm{Kill}(h))^\perp$.
Now
put 
$$Q=-\sum_{i=1}^n\left(\langle B_{\xi(e_i)}B_{\xi(e_i)}\tau(\xi),u\rangle_\h+
\langle K^\mathfrak{h}(\tau(\xi),\xi(e_i))\xi(e_i),u\rangle_\h\right)+\langle B_{\xi(U^\G)}\tau(\xi),u\rangle_\h.$$We have\small{
\begin{eqnarray*}
\sum_{i=1}^n\langle K^\mathfrak{h}(\tau(\xi),\xi(e_i))\xi(e_i),u\rangle_\h&=&
\sum_{i=1}^n\left(\langle B_{\tau(\xi)}B_{\xi(e_i)}\xi(e_i)  ,u\rangle_\h
-\langle B_{\xi(e_i)}B_{\tau(\xi)}\xi(e_i)  ,u\rangle_\h
-\langle B_{[\tau(\xi),\xi(e_i)]^\h}\xi(e_i),u\rangle_h  \right)\\
&=&\langle B_{\tau(\xi)}U^\xi,u\rangle_\h-
\sum_{i=1}^n\left(
\langle B_{\xi(e_i)}B_{\tau(\xi)}\xi(e_i)  ,u\rangle_\h
+\langle B_{[\tau(\xi),\xi(e_i)]^\h}\xi(e_i),u\rangle_h  \right).
\end{eqnarray*}}So\small{
\begin{eqnarray*}
Q&=&
-\langle B_{\tau(\xi)}U^\xi,u\rangle_\h+\langle B_{\xi(U^\G)}\tau(\xi),u\rangle_\h
-\sum_{i=1}^n\left(\langle B_{\xi(e_i)}[\xi(e_i),\tau(\xi)]^\h,u\rangle_\h-
\langle B_{[{\tau(\xi)},{\xi(e_i)}]^\h}\xi(e_i),u\rangle_\h\right)\\
&=&
-\langle B_{\tau(\xi)}U^\xi,u\rangle_\h+\langle B_{U^\xi-\tau(\xi)}\tau(\xi),u\rangle_\h
-\sum_{i=1}^n\left(\langle [{\xi(e_i)},[\xi(e_i),\tau(\xi)]^\h]^\h,u\rangle_\h-2
\langle B_{[{\tau(\xi)},{\xi(e_i)}]^\h}\xi(e_i),u\rangle_\h\right)\\
&\stackrel{\eqref{lc}}=&-\langle [u,\tau(\xi)]^\h,\tau(\xi)\rangle_\mathfrak{h}-\langle
[{\tau(\xi)},U^\xi]^\h,u\rangle_\mathfrak{h}\\&&
-\sum_{i=1}^n\left(\langle [{\xi(e_i)},[\xi(e_i),\tau(\xi)]^\h]^\h,u\rangle_\h
-\langle [{[{\tau(\xi)},{\xi(e_i)}]^\h},\xi(e_i)]^\h,u\rangle_\h\right.\\&&
\left.-\langle [{u},\xi(e_i)]^\h,[{\tau(\xi)},{\xi(e_i)}]^\h\rangle_\h
-\langle [u,[{\tau(\xi)},\xi(e_i)]^\h]^\h,{\xi(e_i)}\rangle_h
\right)\\
&=&-\langle [u,\tau(\xi)]^\h,\tau(\xi)\rangle_\mathfrak{h}-\langle
[{\tau(\xi)},U^\xi]^\h,u\rangle_\mathfrak{h}
+\sum_{i=1}^n\left(
\langle [{u},\xi(e_i)]^\h,[{\tau(\xi)},{\xi(e_i)}]^\h\rangle_h
+\langle [u,[{\tau(\xi)},\xi(e_i)]^\h]^\h,{\xi(e_i)}\rangle_h
\right)\\
&=&\tr(\xi^*\circ(\ad_u+\ad_u^*) \circ\ad_{\tau(\xi)} \circ\xi)
-\langle [u,\tau(\xi)]^\h,\tau(\xi)\rangle_\mathfrak{h}-\langle
[{\tau(\xi)},U^\xi]^\h,u\rangle_\mathfrak{h}.
\end{eqnarray*}}So we get the second relation.
\end{proof}
As an immediate consequence of this proposition we get the following result.
\begin{proposition}\label{prbi} Let $\phi:G\too H$ be an homomorphism between two Riemannian Lie groups. Then:
\begin{enumerate}\item[$(i)$] If the metric on $G$ is bi-invariant and $\phi$ is a submersion then $\phi$ is harmonic.
\item[$(ii)$] If the metric on $H$ is bi-invariant  then $\phi$ is biharmonic, it is harmonic when $\G$ is unimodular.

\end{enumerate}

\end{proposition}

\begin{proof}\begin{enumerate}\item[$(i)$]  Since $\xi$ is an homomorphism of Lie algebras, for any $u\in\G$, $\xi\circ\ad_{u}=\ad_{\xi(u)}\circ \xi$ and hence
\[ \langle U^\xi,\xi(u)\rangle_\h=\tr(\xi^*\circ\ad_{\xi(u)}\circ \xi)=
\tr(\xi^*\circ\xi\circ\ad_{u})=0,
 \] since $\ad_u$ is skew-symmetric and $\xi^*\circ\xi$ is symmetric. Thus $U^\xi=0$. Now we have also $U^\G=0$ and hence $\tau(\xi)=0$.
 \item[$(ii)$] If the metric on $H$ is bi-invariant then $\mathrm{Kill}(h)=\h$ and hence, according to Proposition \ref{pr2}, $U^\xi=0$. By using the expression of $\tau_2(\xi)$ given in Proposition \ref{pr2}, one can see easily that $\tau_2(\xi)=0$.
\end{enumerate}

\end{proof}

\begin{example}\label{exem} There are some interesting situations where we can apply  Proposition \ref{prbi}.\begin{enumerate}\item
 Let $H$ be a compact connected semisimple Lie group and $\pi:G\too H$ a covering homomorphism of $H$ by a Lie group $G$. Then $G$ is compact and hence unimodular.  Then, for any left invariant Riemannian $g$ on $G$ and any bi-invariant Riemannian metric $h_0$ on $H$, $\pi:(G,g)\too(H,h_0)$ is harmonic. Moreover, for any left invariant Riemannian metric $h$ on $H$ and any bi-invariant Riemannian metric $g_0$ on $G$, $\pi:(G,g_0)\too (H,h)$ is harmonic.
 \item Let $G$ be a compact Lie group and $\rho:G\too\mathrm{GL}(V,\R)$ a finite representation of $G$. Then there exists a definite positive product $\prs$ on $V$ which is $G$-invariant, thus $\rho:G\too \mathrm{SO}(V,\prs)$. Now $\mathrm{SO}(V,\prs)$ carries a bi-invariant Riemannian metric $k$ and hence for any left invariant Riemannian metric $g$ on $G$, $\rho:(G,g)\too(\mathrm{SO}(V,\prs),k)$ is harmonic.
 \item Let $(G,g)$ be a compact Lie group endowed with a bi-invariant Riemannian metric. For any closed normal subgroup $N$ of $G$, $\pi:(G,g)\too (G/N,h)$ is harmonic where $h$ is any left invariant Riemannian metric on $G/N$.
\item Let $G$ and $H$ two Lie groups where $H$ is compact and $\al:H\too\mathrm{Aut}(G)$ an homomorphism. Consider the semi-direct product $G\times_\al H$. If $h$ is a bi-invariant Riemannian metric on $H$ and $g$ any left invariant metric on $G\times_\al H$ then $(G\times_\al H,g)\too (H,h)$ is biharmonic.
 \end{enumerate}
\end{example}

\begin{remark}

Actually, one can define  harmonicity and biharmonicity of homomorphisms $\phi:G\too H$ where $G$ is a Riemannian Lie group and $H$ is an affine Lie group, i.e., a Lie group with a left invariant connection. In this case, in the expressions of $\tau(\xi)$ and $\tau_2(\xi)$ given in Proposition \ref{pr1}, $B$ is  the product associated to the left invariant  connection on $H$. The fact is that, by taking the bi-invariant connection $\na$ on $H$ called neutral and given by $\na_XY=\frac12[X,Y]$ for any left invariant vector fields $X,Y$, on can see easily that $\phi$ is biharmonic and it is harmonic if $G$ is unimodular. This generalizes $(ii)$ of Proposition \ref{prbi}. One can see \cite{jap} where harmonicity  of maps into affine Lie groups  is investigated.

\end{remark}

Recall that a K\"ahlerian Lie group is a Lie group $G$ endowed with a left invariant K\"ahler structure. This is equivalent to the existence on $\G$ of a complex structure $J:\G\too\G$ and an Euclidean product $\prs$ such that, for any $u,v\in\G$,
\[ \langle Ju,Jv\rangle=\langle u,v\rangle\esp A_uJv=JA_uv, \]where $A$ is the Levi-Civita associated to $(\G,\prs)$. An homomorphism $\phi:(G,g,J)\too (H,h,K)$ between two K\"ahlerian Lie groups is holomorphic iff, for any $u\in\G$, $\xi(Ju)=K\xi(u)$.
The following result is a particular case of a general well-known result (see \cite{lichne}).

\begin{proposition}\label{prkahler} Let $\phi:(G,g,J)\too (H,h,K)$ be an homomorphism between two K\"ahlerian Lie groups. If $\phi$ is holomorphic then $\phi$ is harmonic.

\end{proposition}

\begin{proof} There exists an orthonormal basis of $\G$ having the form $(e_i,Je_i)_{i=1}^n$. We have
\begin{eqnarray*}
\tau(\xi)&=&\sum_{i=1}^n(B_{\xi(e_i)}\xi(e_i)-\xi(A_{e_i}e_i))+
\sum_{i=1}^n(B_{\xi(Je_i)}\xi(Je_i)-\xi(A_{Je_i}Je_i))\\
&=&\sum_{i=1}^n(B_{\xi(e_i)}\xi(e_i)-\xi(A_{e_i}e_i))+
\sum_{i=1}^n(B_{K\xi(e_i)}K\xi(e_i)-\xi(JA_{Je_i}e_i))\\
&=&\sum_{i=1}^n(B_{\xi(e_i)}\xi(e_i)-\xi(A_{e_i}e_i))+
\sum_{i=1}^n(KB_{K\xi(e_i)}\xi(e_i)-K\xi(A_{Je_i}e_i))\\
&=&\sum_{i=1}^n(B_{\xi(e_i)}\xi(e_i)-\xi(A_{e_i}e_i))+
\sum_{i=1}^n(K[{K\xi(e_i)},\xi(e_i)]\\&&+KB_{\xi(e_i)}K\xi(e_i)
-K\xi([{Je_i},e_i]))-K\xi(A_{e_i}Je_i)\\
&=&\sum_{i=1}^n(B_{\xi(e_i)}\xi(e_i)-\xi(A_{e_i}e_i))+
\sum_{i=1}^n(K[{\xi(Je_i)},\xi(e_i)]-B_{\xi(e_i)}\xi(e_i)
-K[\xi({Je_i}),\xi(e_i)]+\xi(A_{e_i}e_i))\\
&=&0.\qedhere
\end{eqnarray*}
\end{proof}

\subsection{Riemannian immersions between Riemannian Lie groups}

It is a well known result that a Riemannian immersion is harmonic iff it is minimal (see \cite{eells}). We recover this fact in our context in an easy way.
Let $(\G,\prs_\G)$ be an Euclidean Lie algebra and $\G_0$ a subalgebra of $\G$. If $A$ is the Levi-Civita product of $(\G,\prs_\G)$, then for any $u,v\in\G_0$,
\begin{equation*}
A_uv=A_u^0v+h(u,v),
\end{equation*}where $A^0$ is the Levi-Civita product of $(\G_0,\prs_{\G_0})$ ($\prs_{\G_0}$ is the restriction of $\prs_{\G}$ to $\G_0$) and $h:\G_0\times\G_0\too\G_0^\perp$ is bilinear symmetric. It is called the second fundamental form and its   trace  with respect to $\prs_\G$ is the vector $H^{\G_0}\in\G_0^\perp$  given by
\begin{equation*}\label{H}
H^{\G_0}=\sum_{i=1}^nh(e_i,e_i),
\end{equation*}where $(e_i)$ is an orthonormal basis of $\G_0$. This vector is called the mean curvature vector of the inclusion of $\G_0$ in 
$(\G,\prs_\G)$. So we get the following proposition.
\begin{proposition}\label{pr3} Let $\phi:G\too H$ be an homomorphism between two Riemannian Lie groups which is also a Riemannian immersion. Then $\phi$ is harmonic iff $H^{\xi(\G)}=0$.

\end{proposition}

The following two propositions can be used to build examples of minimal Riemannian immersions.
\begin{proposition}\label{pr4} Let $\phi:G\too H$ be an homomorphism between two Riemannian Lie groups. Suppose that $\phi$ is a Riemannian immersion, both $\G$ and $\h$ are unimodular and $\dim H=\dim G+1$. Then $\phi$ is harmonic.

\end{proposition}
\begin{proof} Choose an orthonormal basis $(e_1,\ldots,e_n)$ of $\G$ and complete by $f$ to get an orthonormal basis $(\xi(e_1),\ldots,\xi(e_n),f)$ of $\h$. On the other hand, we have $\tau(\xi)=H^{\xi(\G)}=\al f$. We have, by using \eqref{ug},
\[ \tau(\xi)=\sum_{i=1}^nB_{\xi(e_i)}\xi(e_i)=U^\h-B_ff=-B_ff. \]
So $\langle \tau(\xi),\tau(\xi)\rangle_\h=-\al\langle B_ff,f\rangle_\h\stackrel{\eqref{lc}}=0$. Thus $\tau(\xi)=0$.
\end{proof}

\begin{proposition}\label{pr4bis} Let $\phi:G\too H$ be an homomorphism between two Riemannian Lie groups. Suppose that $\phi$ is a Riemannian immersion,  $[\G,\G]=\G$, any derivation of $\G$ is inner and $\xi(\G)$ is an ideal of $\h$. Then $\phi$ is harmonic.

\end{proposition}

\begin{proof} Choose an orthonormal basis $(e_i)_{i=1}^n$ of $\G$. By using Proposition \ref{pr2}, we get for any $u\in\xi(\G)^\perp$,
\begin{eqnarray*}
\langle H^{\xi(\G)},u\rangle_\h&=&\sum_{i=1}^n\langle\ad_u\xi(e_i),\xi(e_i)\rangle_\h=\tr(\wi\ad_u),
\end{eqnarray*}where $\wi\ad_u$ is the restriction of $\ad_u$ to $\xi(\G)$. Now $\xi(\G)$ being an ideal, $\wi\ad_u$ is a derivation of $\xi(\G)$ an hence from the hypothesis it is inner  and $\tr(\wi\ad_u)=0$. Finally, $H^{\xi(\G)}=0$ which proves the proposition.
\end{proof}

\begin{remark}\label{rem2}\begin{enumerate}\item The class of Lie algebras $\G$ satisfying $[\G,\G]=\G$ and any derivation of $\G$ is inner contains the class of semi-simple Lie algebras and, actually, it is more large than this subclass (see \cite{benayadi}).

 \item In Proposition \ref{pr4}, if $H$ is nilpotent then both $\G$ and $\h$ are unimodular. This can be used to construct many examples of minimal Riemannian immersions into Riemannian nilmanifolds. For instance let $\h$ be the 5-dimensional nilpotent Lie algebra whose Lie brackets are given by 
\[ [e_1,e_2]=e_3,\;[e_1,e_3]=e_5,\;[e_2,e_4]=e_5. \] $\G=\spa\{e_1,e_2,e_3,e_5\}$ is a subalgebra of $\h$. If $G$ is the connected and simply connected Lie group associated to $\G$ and $H$ the connected subgroup associated to $\h$ then, for any left invariant Riemannian metric on $G$, the inclusion $H\too G$ is a minimal Riemannian immersion. Moreover, according to Malcev's Theorem (see \cite{raghun}) $G$ has uniform lattices, i.e., there exists a discrete subgroup $\Ga$ of $G$ such that $\Ga/G$ is compact. Thus we get a minimal immersion into a compact nilmanifold $H\too \Ga/G$. 
\end{enumerate}

\end{remark}

 \section{Harmonic automorphisms of a Riemannian Lie group}\label{section3}

 Denote by $\mathcal{M}^\ell(G)$ the set of all left invariant Riemannian metrics on a Lie group $G$ and fix  $g\in \mathcal{M}^\ell(G)$. Recall that $\mathrm{Inisom}(g)$ is the subgroup of $G$ consisting of $a\in G$ such that $i_a=L_a\circ R_{a^{-1}}$ is an isometry, $\mathrm{Kill}(g)=\{u\in\G,\ad_u+\ad_u^*=0\}$ its Lie algebra. We denote by $H(g)$ the set consisting of $a\in G$ such that $i_a$ is
 harmonic. We have obviously $Z(G)\subset \mathrm{Inisom}(g)\subset H(g)$, $\mathrm{Inisom}(g)H(g)\subset H(g)$ and $H(g)\mathrm{Inisom}(g)\subset H(g)$ . Denote by $\pi:G\too G/\mathrm{Inisom}(g)$ the natural projection. Note that if $\phi$ is an automorphism de $G$, $H(\phi^*g)=\phi^{-1}(H(g))$. The set $H(g)$ has been investigated first by Park in \cite{park} in the case of a semisimple compact Lie group. In particular, he determined $H(g)$ for any left invariant Riemannian metric on $\mathrm{SU}(2)$. If one look carefully to the result of Park, one can see easily, in the case of $\mathrm{SU}(2)$, that the cardinal of
 $H(g)/\mathrm{Inisom}(g)$ is finite. We will show now that this result is true on any compact Lie group. This is based on the following lemma.
  
 \begin{lemma}\label{park1}If $G$ is unimodular then  $\mathrm{Inisom}(g)$ is open in $H(g)$, i.e., there exists an open set $U\subset G$ such that $U\cap H(g)=\mathrm{Inisom}(g)$. In particular, the quotient topology on $H(g)/\mathrm{Inisom}(g)$ is discrete.
 
 \end{lemma}
 
 \begin{proof} Since $G$ is unimodular, according to the first relation in Proposition \ref{pr2},  $a\in H(g)$ iff
  \[ \forall u\in\G,\; \tr(\mathrm{Ad}_a^*\circ\ad_u\circ
  \mathrm{Ad}_a)=0.\]Define
  $\al:G\too\G^*$ by
  $\al(a)(u)=\tr(\mathrm{Ad}_a^*\circ\ad_u\circ
  \mathrm{Ad}_a).$
 Thus $H(g)=\al^{-1}(0)$. The differential of $\al$ at $a\in H(g)$ is given by
   \begin{equation}\label{eqdiff} d_a\al(T_eR_a(u))(v)=\tr(\mathrm{Ad}_a^*\circ\ad_u^*\circ\ad_v\circ
   \mathrm{Ad}_a)+\tr(\mathrm{Ad}_a^*\circ\ad_v\circ\ad_u\circ
    \mathrm{Ad}_a). \end{equation}If $a\in \mathrm{Inisom}(g)$ then $\Ad_a^*=\Ad_{a^{-1}}$ and hence
    \[ d_a\al(T_eR_a(u))(v)=\tr((\ad_u^*+\ad_u)\circ\ad_v). \]
    So $T_eR_a(u)\in\ker d_a\al$ iff, for any $v\in \G$, $\tr((\ad_u^*+\ad_u)\circ\ad_v)=0$. In particular, we get $\tr((\ad_u^*+\ad_u)\circ\ad_u)=0$. By using the properties of the trace we get also $\tr((\ad_u^*+\ad_u)\circ\ad_u^*)=0$ and hence $\ad_u^*+\ad_u=0$. Thus $\ker d_a\al=T_eR_a(\mathrm{Kill}(g))$. Now, it is easy to see that $\al$ factor to give a smooth map $\wi\al:G/\mathrm{Inisom}(g)\too \G^*$ and from what above, we get that $\wi\al$ is an immersion at $\pi(e)$ and hence there exists an open set $\wi U\subset G/\mathrm{Inisom}(g)$ containing $\pi(e)$ such that the restriction of $\wi\al$ to $\wi U$ is injective. Thus  $U=\pi^{-1}(\wi U)$ satisfies the conclusion of the theorem.  
  \end{proof}
 If $G$ is compact then $\mathrm{Inisom}(g)$ and $H(g)$ are compact, so we get the following  interesting result.
 \begin{theorem}\label{finite} If $G$ is compact then $H(g)/\mathrm{Inisom}(g)$ has a finite cardinal.
 
 \end{theorem}
 
 In \cite{park}, Park showed that if for a left invariant Riemannian metric $g$ on $\mathrm{SU}(2)$ any inner automorphism is harmonic then $g$ is actually bi-invariant. The following theorem generalizes this result to any connected Lie group.
 
 \begin{theorem}\label{park} Let $(G,g)$ be a connected Riemannian Lie group such that $H(g)=G$. Then $g$ is bi-invariant.
 
 \end{theorem}
 \begin{proof}The hypothesis of the theorem is equivalent to:
 \[ \forall a\in G,\;\forall u\in\G,\; \langle U^{\mathrm{Ad}_a},u\rangle_\G-
  \langle U^{\G},\mathrm{Ad}_a^*u\rangle_\G=0.\]According to the first relation in Proposition \ref{pr2} this is equivalent to
 \[ \forall a\in G,\;\forall u\in\G,\; \tr(\mathrm{Ad}_a^*\circ\ad_u\circ
 \mathrm{Ad}_a)=\tr(\ad_{\mathrm{Ad}_a^*u}).\]
 By taking $a=\exp(tv)$ and differentiating this relation, we get
 \[ \forall u,v\in\G,\;\tr((\ad_v+\ad_v^*)\circ\ad_u)=\tr(\ad_{\mathrm{ad}_v^*u}). \]
 Remark that since $U^\G\in[\G,\G]^\perp$ then $\mathrm{ad}_{U^\G}^*U^{\G}=0$ and hence
 \[ \tr((\ad_{U^\G}+\ad_{U^\G}^*)\circ\ad_{U^\G})=0, \]which is equivalent to 
 $\ad_{U^\G}+\ad_{U^\G}^*=0$ and hence $\langle U^\G,U^\G\rangle_\G=\tr(\ad_{U^\G})=0$. Thus $\G$ is unimodular and we can apply Lemma \ref{park1} to conclude.
 \end{proof}
 
 \begin{theorem}\label{theonilpotent}If $G$ is abelian or 2-step nilpotent then $H(g)=\mathrm{Inisom}(g)=Z(G)$.
 
 \end{theorem}
 
 \begin{proof} The theorem is obvious when $G$ is abelian. Suppose now that $G$ is 2-step nilpotent. Then $\exp:\G\too G$ is a diffeomorphism. An element $\exp(u)\in H(g)$ iff,
 \[ \forall v\in\G,\quad\tr(\Ad_{\exp(u)}^*\circ\ad_v\circ\Ad_{\exp(u)})=0. \]
 Now, $\Ad_{\exp(u)}=\exp(\ad_u)=\mathrm{Id}_{\G}+\ad_u$ and for any $v,w\in\G$, $\ad_v\circ\ad_w=0$. So $\exp(u)\in H(g)$ iff
 \[ \forall v\in\G,\quad\tr(\ad_u^*\circ\ad_v)=0. \]
 By taking $v=u$, we get $\ad_u=0$ which achieves the proof.
 \end{proof}
 
 \begin{example} In \cite{park}, Park determined harmonic inner
automorphisms of $(\mathrm{SU}(2), g)$ for every left invariant Riemannian metric $g$. In this example, we consider
  $G=\mathrm{SL}(2,\R)$ and $\G=\mathrm{sl}(2,\R)$ and we give the equations determining harmonic inner automorphisms for a particular class of left invariant Riemannian metrics.
   Put
          \[ h=\left(\begin{array}{cc}1&0\\0&-1  \end{array} \right),\;
          e=\left(\begin{array}{cc}0&1\\0&0  \end{array} \right),\;
          f=\left(\begin{array}{cc}0&0\\1&0  \end{array} \right). \]
          We have
          \[ [e,f]=h,\; [h,e]=2e,\; \esp [h,f]=-2f. \]
          We consider the Euclidean product on $\G$ for which $(h,e,f)$ is orthogonal and $\langle h,h\rangle=\al_1$, $\langle e,e\rangle=\al_2$ and $\langle f,f\rangle=\al_3$ and we denote by $g$ the associated left invariant metric on $G$. Put $\al_{ij}=\al_i(\al_j)^{-1}$.  A direct computation shows that $A=
          \left(\begin{array}{cc}a&b\\c&d\end{array}\right)\in H(g)$ iff 
     \[ \left\{ \begin{array}{ccc}8(a^2b^2\al_{21}-c^2d^2\al_{31})+2(a^4-d^4+b^4\al_{23}-c^4\al_{32})&=&0,\\ 
     2(ad+bc)(2ab\al_{21}+cd)+ac(c^2\al_{12}+2a^2)+bd(d^2\al_{13}+2b^2\al_{23})&=&0,\\
     2(ad+bc)(ab+2cd\al_{31})+ac(a^2\al_{12}+2c^2\al_{32})+bd(b^2\al_{13}+2d^2)&=&0.\end{array}   \right. \]

 \end{example}

 We consider now the following problem. Given a Riemannian Lie group $(G,g)$ one can aim to determine all the couples $(\phi,h)$ where $\phi$ is an automorphism of $G$ and $h\in \mathcal{M}^\ell(G)$  such that $\phi:(G,g)\too(G,h)$ is harmonic. By remarking that $\phi:(G,g)\too(G,h)$ is harmonic iff $\mathrm{Id}_G:(G,g)\too(G,\phi^*h)$ is harmonic, the solution of the problem is equivalent to the determination of the group $\mathrm{Aut}(G)$ and the set $CH(g)$ of the left invariant Riemannian metric $h$ on $G$ such that
 $\mathrm{Id}_G:(G,g)\too(G,h)$ is harmonic. 
 
\begin{proposition}Let $(G,g)$ be a Riemannian Lie group. Then $h\in CH(g)$ 
iff, for any $u\in\G$, 
\begin{equation}\label{identity} 
\tr(J\circ\ad_{u})=\tr(\ad_{Ju}),
 \end{equation}where $J$ is given by
$ h(u,v)=g (Ju,v)$ for any $u,v\in\G$. In particular, $CH(g)$ is a convex cone which contains $g$.

\end{proposition}

\begin{proof}Denote $\prs_1=g(e)$,  $\prs_2=h(e)$, $A$ the Levi-Civita product of $(\G,\prs_1)$ and $B$ the Levi-Civita product of $(\G,\prs_2)$.
 We have, for any $u\in\G$, and for any orthonormal basis $(e_i)_{i=1}^n$ of $\prs_1$
\begin{eqnarray*}
\langle\tau(\mathrm{Id}_\G),u\rangle_2&=&\sum_{i=1}^n\langle B_{e_i}e_i,u\rangle_2-\sum_{i=1}^n\langle A_{e_i}e_i,u\rangle_2\\
&=&\sum_{i=1}^n\langle [u,e_i],e_i\rangle_2-\sum_{i=1}^n\langle A_{e_i}e_i,Ju\rangle_1\\
&=&\tr(J\circ\ad_{u})-\tr(\ad_{Ju}).
\end{eqnarray*}\end{proof}

 \begin{definition}\label{harmonicdimension}
 We call $CH(g)$ the {\it harmonic cone} of $g$ and  $\dim CH(g)$ the {\it harmonic dimension} of $g$, where $\dim CH(g)$ is the dimension of the subspace spanned by $CH(g)$. \end{definition} This is an invariant of $g$ in the following sense.
  If $\phi:(G,g_1)\too(G,g_2)$ is an automorphism such that $\phi^*g_2=\al g_1$ with $\al$ is positive constant then $CH(g_2)=(\phi^{-1})^*CH(g_1)$ and hence $\dim CH(g_2)=\dim CH(g_1)$.
The following proposition is a consequence of Theorem \ref{park}.
\begin{proposition}\label{full} Let $(G,g)$ be a Riemannian Lie group. Then  $CH(g)=\mathcal{M}^\ell(G)$  iff $g$ is bi-invariant.

\end{proposition}
\begin{proof} If $g$ is bi-invariant then, according to Proposition \ref{prbi}, for any left invariant metric $h$, $\mathrm{Id}_G:(G,g)\too(G,h)$ is harmonic. Suppose now that $CH(g)$ contains all the left invariant Riemannian metrics on $G$. Then, for any $a\in G$, $\Ad_a^*(g)\in CH(g)$ and hence $\Ad_a:(G,g)\too(G,g)$ is harmonic. By applying Theorem \ref{park} we get that $g$ is bi-invariant.
\end{proof}

\begin{example}\begin{enumerate}\item Let $E(1)$ be the 2-dimensional Lie group of rigid motions of the real line. Then, according to Proposition \ref{pr2d}, for any left invariant Riemannian metric $g$ on $E(1)$, $CH(g)=\left\{\al g,\al>0  \right\}$. Thus the harmonic dimension of any left invariant Riemannian metric on $E(1)$ is equal to 1.
 \item Let $H_3$ be the   3-dimensional Heisenberg  group and $g$ a left invariant Riemannian metric on $g$.  Denote by $\h_3$ its Lie algebra and $\prs_1=g(e)$. There exists an $\prs_1$-orthonormal basis $(z,f,g)$ such that $[f,g]=\al z$. A direct computation solving \eqref{identity} shows that $h\in CH(g)$ iff  $h(u,v)(e)=\langle Ju,v\rangle_1$ where the matrix of $J$ in the basis $(z,f,g)$ has the form
 $$
   \left(\begin{array}{ccc}a&0&0\\0&d&e\\0&e&h  \end{array} \right),\quad a>0,
   d+h>0\esp dh-e^2>0.$$Thus the harmonic dimension of any left invariant Riemannian metric on $H_3$ is equal to 4.
   \item Let $G=\mathrm{SO}(3,\R)$ and $\G=\mathrm{so}(3)$ its Lie algebra. Fix $g$ a left invariant Riemannian metric on $G$ and denote by $\prs=g(e)$. Denote also by $\prs_0$ the bi-invariant Euclidean product given
   \[ \langle A,B\rangle_0=-\tr(AB). \]
   Define $J_0$ by $\langle u,v\rangle=\langle J_0u,v\rangle_0$. There exists
    an $\prs_0$-orthonormal basis of $(X_1,X_2,X_3)$ of $\G$ such that
    $ J_0X_i=\al_i X_i$ with $i=1,2,3$ and $\al_i>0$. Since $\prs_0$ is bi-invariant it is easy to see that there exists a constant $c$ such that
    
    \[ [X_1,X_2]=cX_3,\;
    [X_2,X_3]=cX_1,\;[X_3,X_1]=cX_2\esp\langle X_i,X_j\rangle_1=\de_{ij}\al_i. \] Denote by $M$ the matrix of $\prs_1$ in this basis.
    By identifying an endomorphism with its matrix in $(X_1,X_2,X_3)$, we have
      \[ \ad_{X_1}=\left(\begin{array}{ccc}0&0&0\\0&0&-c\\0&
      c&0  \end{array} \right),\;
      \ad_{X_2}=\left(\begin{array}{ccc}0&0&c\\0&0&0\\
      -c&0&0  \end{array} \right),\;\]\[
       \ad_{X_3}=\left(\begin{array}{ccc}0&-c&0\\
       c&0&0\\0&0&0  \end{array} \right)\esp J=
       \left(\begin{array}{ccc}a&b&c\\b'&d&e\\c'&e'&f  \end{array} \right).\]
       Then \eqref{identity} is equivalent to $b=b'$, $c=c'$ and $e=e'$. The condition that $J$ is symmetric with respect to $\prs_1$ is equivalent to $MJ=JM$ which is equivalent to
       \[ (\al_1-\al_2)b=(\al_3-\al_1)c=(\al_3-\al_2)e=0. \]
       If the $\al_i$ are distinct then $\dim CH(g)=3$. If $\al_i=\al_j\not=\al_k$ then $\dim CH(g)=5$. If $\al_1=\al_2=\al_3$ then $g$ is bi-invariant and $\dim CH(g)=6$.

   \end{enumerate}
   
   \end{example}

 On all the examples above, we have the following formula
$$\dim CH(g)=\frac{n(n-1)}2+\dim \mathrm{Kill}(g),$$where $n$ is the dimension of the Lie group $G$. We will show now that this formula is valid in the general case when the group is unimodular.
\begin{proposition}\label{compute}
Let $(G,g)$ be an unimodular Riemannian  Lie group. Then
$$\dim CH(g)=\frac{n(n-1)}2+\dim \mathrm{Kill}(g),$$where $n$ is the dimension of the Lie group $G$.\end{proposition}
\begin{proof}
 If  $(\G,\prs)$ is the  Lie algebra of $G$, the metric $\prs$ defines a definite positive product on $\mathrm{gl}(\G)$ by $\langle A,B\rangle_1=\tr(A^*B)$.
Define $\phi:\G/\mathrm{Kill}(g)\too\mathrm{gl}(\G)$, $[u]\mapsto \ad_u+\ad_u^*$. This is into and by using \eqref{identity}, we get
\[ CH(g)=\phi\left(\G/\mathrm{Kill}(g) \right)^\perp\cap\mathrm{Sym}^+(\G), \]
where $\phi\left(\G/\mathrm{Kill}(g) \right)^\perp$ is the orthogonal of $\phi\left(\G/\mathrm{Kill}(g) \right)$ with respect to $\prs_1$ and $\mathrm{Sym}^+(\G)$ is the convex cone of positive definite symmetric isomorphisms of $(\G,\prs)$. Since $\mathrm{Sym}^+(\G)$ is open in the space of symmetric endomorphisms $\mathrm{Sym}(\G)$, $CH(g)$ is open in $\phi\left(\G/\mathrm{Kill}(g) \right)^\perp\cap\mathrm{Sym}(\G)$. Thus
\[ \dim CH(g)=\dim\phi\left(\G/\mathrm{Kill}(g) \right)^\perp\cap\mathrm{Sym}(\G)=
\frac{n(n-1)}2+\dim \mathrm{Kill}(g).
 \]\end{proof}

\section{Biharmonic submersions between Riemannian Lie groups}\label{section4}

Let $\phi:(G,g)\too (H,h)$ be a submersion between two Riemannian Lie groups. Then $G_0=\ker\phi$ is a normal subgroup of $G$, $G/G_0$ is a Lie group and $\overline{\phi}:G/G_0\too H$ is an isomorphism. Let $\pi:\G\too\G/\G_0$ the natural projection.
If $\xi:\G\too \h$ is the differential of $\phi$ at $e$, the restriction of $\pi$ to $\ker\xi^\perp$ is an isomorphism onto $\G/\G_0$ and we denote by $\mathrm{r}:\G/\G_0\too\ker\xi^\perp$ its inverse. Thus $\mathrm{r}^*\prs_\G$ is  an Euclidean product on $\G/\G_0$ which defines a left invariant Riemannian metric $g_0$ on $G/G_0$. We denote by $\overline{\xi}$ the differential of $\overline{\phi}$ at $e$. 
\begin{proposition}\label{submersion} With the notations above, we have
\[ \tau(\xi)=\tau(\overline{\xi})-\xi(H^{\ker\xi}), \]where $\overline{\xi}:(\G/\G_0,\mathrm{r}^*\prs_\G)\too(\h,\prs_\h)$.

\end{proposition}

\begin{proof} We have $\G=\ker\xi\oplus\ker\xi^\perp$. Choose an orthonormal basis $(f_i)_{i=1}^p$ of $\ker\xi$ and an orthonormal basis $(e_i)_{i=1}^q$ of $\ker\xi^\perp$. If $A$ and $B$ denote the Levi-Civita products of $\G$ and $\h$ respectively, we have
\[ \tau(\xi)=\sum_{i=1}^qB_{\xi(e_i)}\xi(e_i)-\sum_{i=1}^q\xi(A_{e_i}e_i)
-\sum_{i=1}^p\xi(A_{f_i}f_i). \]
If we put, for any $u,v\in\ker\xi$, $A_uv=A_u^0v+h(u,v)$ where $A^0$ is the Levi-Civta product of $\ker\xi$, we get
\[ \sum_{i=1}^p\xi(A_{f_i}f_i)=\xi(H^{\ker\xi}). \]
Denote by $\pi:\G\too \G/\G_0$ the natural project. Then $(\pi(e_i))_{i=1}^q$ is an orthonormal basis of $\G/\G_0$ and hence
\[ U^\xi=\sum_{i=1}^qB_{\xi(e_i)}\xi(e_i)=
\sum_{i=1}^qB_{\overline{\xi}(\pi(e_i))}\overline{\xi}(\pi(e_i))=U^{\overline{\xi}}. \]
To achieve the proof, we must show that
\[ \overline{\xi}(U^{\G/\G_0})=\sum_{i=1}^q\xi(A_{e_i}e_i). \]This is a consequence of more general formula. If $\overline{A}$ is the Levi-Civita product on $\G/\G_0$, then for any $u,v\in\ker\xi^\perp$,  
$\pi(A_uv)=\overline{A}_{\pi(u)}\pi(v)$. To  establish this relation note first that, for any $u,v\in\ker\xi^\perp$, we have $[u,v]^\G=\mathrm{r}([\pi(u),\pi(v)]^{\G/\G_0})+\om(u,v)$ where $\om(u,v)\in\ker\xi$. Now, for any $u,v,w\in\ker\xi^\perp$, we have
\begin{eqnarray*}
2\langle {\overline{A}}_{\pi(u)}\pi(v),\pi(w)\rangle_{\G/\G_0}&=&
\langle[\pi(u),\pi(v)]^{\G/\G_0},\pi(w)\rangle_{\G/\G_0}+
\langle[\pi(w),\pi(u)]^{\G/\G_0},\pi(v)\rangle_{\G/\G_0}\\&&+
\langle[\pi(w),\pi(v)]^{\G/\G_0},\pi(u)\rangle_{\G/\G_0}\\
&=& \langle\mathrm{r}([\pi(u),\pi(v)]^{\G/\G_0}),w\rangle_{\G}+
\langle\mathrm{r}([\pi(w),\pi(u)]^{\G/\G_0}),v\rangle_{\G}+
\langle\mathrm{r}([\pi(w),\pi(v)]^{\G/\G_0}),u\rangle_{\G}\\
&=&\langle[u,v]^{\G},w\rangle_{\G}+
\langle[w,u]^{\G},v\rangle_{\G}+
\langle[w,v]^{\G},u\rangle_{\G}\\
&=&2\langle A_uv,w\rangle_\G\\
&=&2\langle \pi(A_uv),\pi(w)\rangle_{\G/\G_0}.
\end{eqnarray*}

\end{proof}

The following proposition is an immediate consequence of Proposition \ref{submersion}.
\begin{proposition}\label{prsubmersion} Let $\phi:(G,g)\too (H,h)$ be a submersion between two Riemannian Lie groups. Then:
\begin{enumerate}\item [$(i)$] If $\ker\xi$ is minimal then $\phi$ is harmonic $($resp. biharmonic$)$ iff $\overline{\phi}$ is harmonic $($resp. biharmonic$)$.
\item[$(ii)$] If $\overline{\phi}$ is harmonic then $\phi$ is harmonic iff $\ker\xi$ is minimal.

\end{enumerate} \end{proposition}

Let $\phi:(G,g)\too (H,h)$ be a submersion between two connected Riemannian Lie groups. The connectedness implies that $\phi$ is onto and $\overline{\phi}:G/G_0\too H$ is an isomorphism. So $\phi$ is harmonic (resp. biharmonic) iff $\Pi:(G,g)\too (G/G_0,\overline{\phi}^*h)$ is harmonic (resp. biharmonic). So the study of harmonic or biharmonic submersion between two connected Riemannian Lie groups is equivalent to the study of the projections $\Pi:(G,g)\too (G/G_0,h)$ where $(G,g)$ is a connected Lie group,  $G_0$ is a normal subgroup and $h$ is left invariant Riemannian metric on $G/G_0$. To build harmonic or biharmonic such projections, let first understand how $G$ can be constructed from $G/G_0$ and $G_0$.\\
Fix $\Pi:(G,g)\too (H,h)$ where $H=G/G_0$,   denote by $\pi:\G\too\G/\G_0$ the natural projection and
 $\mathrm{r}:\h\too(\ker\xi)^\perp$ the inverse of the restriction of $\pi$ to $(\ker\xi)^\perp$.  In this context the formula in Proposition \ref{submersion} has the following simpler form:
 \begin{equation}\label{eqsubmersion}
 \tau(\pi)=\tau(\mathrm{Id}_\h)-\pi(H^{\ker\xi}),
 \end{equation}where $\mathrm{Id}_\h:(\h,\prs_\pi)\too(\h,\prs_\h)$ where $\prs_\pi=\mathrm{r}^*\prs_\G$.

 For any $u\in\G$ we denote by $\wi\ad_u$ the restriction of $\ad_u$ to $\ker\xi$.
 We define $\rho:\h\too\mathrm{Der}(\ker\xi)$ and $\om\in\wedge^2\h^*\otimes\ker\xi$ by
 \begin{equation}\label{omega} \rho(h)=\wi\ad_{\mathrm{r}(h)} \esp \om(h_1,h_2)=[\mathrm{r}(h_1),\mathrm{r}(h_2)]^\G-\mathrm{r}([h_1,h_2]^\h),
 \end{equation}
 where $\mathrm{Der}(\ker\xi)$ is the space of derivations of $\ker\xi$.
 A direct computation using Jacobi identity of $[\;,\;]^\G$ and $[\;,\;]^\h$  shows that
 \begin{equation}\label{condition}
 \rho([h_1,h_2]^\h)=[\rho(h_1),\rho(h_2)]-\wi\ad_{\om(h_1,h_2)}\esp d_\rho\om=0,
 \end{equation}where
 \[ d_\rho\om(h_1,h_2,h_3)=\oint\left(\rho(h_1)(\om(h_2,h_3))
 -\om([h_1,h_2]^\h,h_3) \right). \]The symbol $\oint$ stands for circular permutations. Let give a characterization of $\tau(\xi)$ using the formalism above.
 \begin{proposition}\label{pr5} For any $h\in\h$, we have
 \begin{equation}\label{pr5eq} \langle \pi(H^{\ker\xi}),h\rangle_{\pi}=\tr(\rho(h)). \end{equation}
 \end{proposition}
 \begin{proof}We have  $H^{\ker\xi}=\sum_{i=1}^p(A_{f_i}f_i)-A_{f_i}^0f_i),$ where $(f_i)_{i=1}^p$ is an orthonormal basis of $\ker\xi$, $A$ is the Levi-Civita product of $\G$ and $A^0$ is the Levi-Civita product of $\ker\xi$. So
 \begin{eqnarray*}
 \langle\pi(H^{\ker\xi}),h\rangle_\pi&=&\langle H^{\ker\xi},\mathrm{r}(h)\rangle_\G
 =\sum_{i=1}^p\langle A_{f_i}f_i,\mathrm{r}(h)\rangle_\G
 \stackrel{\eqref{lc}}=\sum_{i=1}^p\langle 
 [\mathrm{r}(h),{f_i}]^\G,f_i\rangle_\G
 =\tr(\rho(h)).  
 \end{eqnarray*}
 \end{proof}

Let study now the converse of our study above. Let $(\n,\h)$ be two  Lie algebras such that $\n$ carries an Euclidean product and $\h$ two Euclidean products $\prs_1$ and $\prs_2$,  $\rho:\h\too\mathrm{Der}(\n)$ and $\om\in\wedge^2\h^*\otimes\n$ satisfying \eqref{condition}. Define on $\G=\n\oplus\h$ the bracket $[\;,\;]^\G$
 \begin{equation}\label{bracket}
 [u,v]^\G=\left\{ \begin{array}{lcl}[u,v]^\n&\mbox{if}&u,v\in\n,\\
 \;[u,v]^\h+\om(u,v)&\mbox{if}&u,v\in\h,\\
 \rho(u)(v)&\mbox{if}&u\in\h,v\in\n.\end{array}
      \right.
 \end{equation}Then $(\G,[\;,\;]^\G,\prs_\G=\prs_\n\oplus\prs_1)$ is an Euclidean Lie algebra and the projection $\pi:\G\too\h$ is an homomorphism of Lie algebras. Let $G$ be the connected and simply connected  Lie group associated to $\G$ and $H$ any connected Riemannian Lie group associated to $\h$. Then there exists a unique homomorphism of Lie groups $\phi:G\too H$ such that $d_e\phi=\pi$. If we endow $G$ and $H$ by the left invariant Riemannian metrics associated respectively to $\prs_\G$ and $\prs_2$, $\phi$ becomes a  submersion. Moreover, $\tau(\phi)(e)$ is given by 
 \begin{equation}\label{eqsubmersion2}
 \tau(\phi)(e)=\tau(\mathrm{Id}_\h)-H^\rho,
 \end{equation}where $\mathrm{Id}_\h:(\h,\prs_1)\too(\h,\prs_2)$ and $H^\rho$ is given by $\langle H^\rho,u\rangle_1=\tr(\rho(u))$ for any $u\in\h$.
 So we have shown the following result.
 \begin{proposition}\label{pr6} There is a correspondence between the set of  submersions with a connected and simply-connected domain and the set of $(\n,\h,\rho,\om)$ where $\n$ is an Euclidean Lie algebra, $\h$ is a Lie algebra having two Euclidean products, $\rho:\h\too\mathrm{Der}(\n)$ and $\om\in\wedge^2\h^*\otimes\n$ satisfying \eqref{condition}.
 
 \end{proposition}

The following proposition is an interesting consequence of \eqref{pr5eq}.
 \begin{proposition} Let $G$ be a connected Riemannian Lie group and $\G_0$ a semisimple normal subgroup of $G$. Then $G_0\subset G$ is minimal and $\Pi:G\too G/G_0$ is harmonic when $G/G_0$ is endowed with the quotient metric $g_0$. Moreover, for any left invariant Riemannian metric $h$ on $G/G_0$, $\Pi:(G,g)\too (G/G_0,h)$ is harmonic $($resp. biharmonic$)$ iff $\mathrm{Id}_{G/G_0}:(G/G_0,g_0)\too (G/G_0,h)$ is harmonic $($resp. biharmonic$)$.
 
 \end{proposition}
 
 \begin{proof} This is a consequence of \eqref{eqsubmersion}, \eqref{pr5eq} and the fact that $\G_0$ being semisimple, for any $u\in\G/\G_0$, the derivation $\rho(u)$ is inner and hence $\tr(\rho(u))=0$.
 \end{proof}

\section{Biharmonic Riemannian submersions between Riemannian Lie groups}\label{section5}

The following proposition follows easily from the last section's study.
\begin{proposition}\label{pr8} Let $\phi:G\too H$ be an homomorphism  between two Riemannian Lie groups which is a Riemannian submersion. Then $\phi$ is harmonic in each of the following cases:
\begin{enumerate}
\item[$(i)$] Both $\G$ and $\h$ are unimodular.
\item[$(ii)$]  $\ker\xi$ is unimodular and the Lie algebra $\h$ of $H$ satisfies $[\h,\h]=\h$.
\item[$(iii)$] $\ker\xi$ satisfies $[\ker\xi,\ker\xi]=\ker\xi$ and $\mathrm{Der}(\ker\xi)=\ad_{\ker\xi}$.
\end{enumerate}

\end{proposition}

\begin{proof}
\begin{enumerate}\item[$(i)$]  This a consequence of the definition of $\tau(\xi)=U^\xi-\xi(U^\G)$ and the fact that $U^\xi=U^\h$ when $\phi$ is a Riemannian submersion.
\item[$(ii)$] According to \eqref{condition} and \eqref{pr5eq}, we have for any $h_1,h_2\in\h$,
\[ \langle\tau(\xi),[h_1,h_2]^\h\rangle_\h=-\tr(\rho([h_1,h_2])=
\tr(\wi\ad_{\om(h_1,h_2)})=0, \] and hence $\phi$ is harmonic.
\item[$(iii)$] From the hypothesis, $\ker\xi$  is unimodular and any derivation of $\ker\xi$ is interior and hence, for any $h\in\h$, $\tr(\rho(h))=0$ and \eqref{pr5eq} gives the result.
\end{enumerate}

\end{proof}

The following proposition gives an useful characterization of biharmonic Riemannian submersions between Riemannian Lie groups.

\begin{proposition}\label{pr7} Let $\phi:G\too H$ be an homomorphism  between two Riemannian Lie groups which is a Riemannian submersion. Then $\phi$ is biharmonic iff one of the following equivalent conditions holds:
\begin{enumerate}
\item[$(i)$] For an orthonormal basis $(e_i)_{i=1}^q$ of $\h$,
\begin{equation}\label{rs1} \sum_{i=1}^qB_{e_i}B_{e_i}\tau(\xi)+\ric^\h(\tau(\xi))-B_{\xi(U^\G)}\tau(\xi)=0, \end{equation}where $\ric^\h$ is the Ricci operator.
\item[$(ii)$] For any $u\in\h$,
\begin{equation}\label{rs2} \tr((\ad_u+\ad_u^*) \circ\ad_{\tau(\xi)} )
-\langle [u,\tau(\xi)],\tau(\xi)\rangle_\mathfrak{h}-\langle
[{\tau(\xi)},U^\h],u\rangle_\mathfrak{h}=0. \end{equation} 
\end{enumerate}

\end{proposition}

\begin{proof} It is a consequence of the fact that $\phi$ is a Riemannian submersion, \eqref{tension} and Proposition \ref{pr2}.
\end{proof}
We can now state this interesting result.
\begin{theorem}\label{theo1} Let $\phi:G\too H$ be an homomorphism  between two Riemannian Lie groups which is a Riemannian submersion. Then:
\begin{enumerate}\item[$(i)$] When $\h$ is unimodular then $\phi$ is biharmonic iff $\tau(\xi)^\ell$ is a Killing vector field. 
\item[$(ii)$] When $\ker\xi$ is unimodular or $\om=0$ then $\phi$ is biharmonic iff $\tau(\xi)^\ell$ is a parallel vector field $($$\omega$ is given by \eqref{omega}$)$. 

\end{enumerate}

\end{theorem}

\begin{proof}\begin{enumerate}\item[$(i)$] Suppose that $\phi$ is biharmonic. By taking $u=\tau(\xi)$ in \eqref{rs2}, we get, since $U^\h=0$,
\[ \tr((\ad_{\tau(\xi)}+\ad_{\tau(\xi)}^*) \circ\ad_{\tau(\xi)} )=0. \]
This equivalent to $\ad_{\tau(\xi)}+\ad_{\tau(\xi)}^*=0$ and hence $\tau(\xi)^\ell$ is a Killing vector field. The converse follows easily from \eqref{rs2}.
\item[$(ii)$] Suppose that $\phi$ is biharmonic. We get from \eqref{rs1}
\[ -\sum_{i=1}^q\langle B_{e_i}\tau(\xi),B_{e_i}\tau(\xi)\rangle_\h+
\langle\ric^\h(\tau(\xi)),\tau(\xi)\rangle_\h=0. \]
By using the fact that $\ker\xi=0$ or $\om=0$, \eqref{condition} and \eqref{pr5eq}, one can see easily that $\tau(\xi)\in[\h,\h]^\perp$. It follows (see \cite{milnor} Lemma 2.3) that
\[ \langle\ric^\h(\tau(\xi)),\tau(\xi)\rangle_\h=-\tr((\ad_{\tau(\xi)}+
\ad_{\tau(\xi)}^*)^2)\leq0. \]
So $B\tau(\xi)=0$ which is equivalent to the fact that $\tau(\xi)^\ell$ is parallel. The converse follows from the fact that if 
$B\tau(\xi)=0$ then $\ad_{\tau(\xi)}+
\ad_{\tau(\xi)}^*=0$.
\end{enumerate}
\end{proof}

Let $H$ be a Riemannian Lie group, the tangent space $TH$ has natural Lie group structure for which the Sasaki metric is left invariant and the projection $\pi:TH\too H$ is a Riemannian submersion. We have the following result.
\begin{proposition} The following assertions are equivalent:
\begin{enumerate}\item[$(i)$] The projection $\pi:TH\too H$ is harmonic.
\item[$(ii)$] The projection $\pi:TH\too H$ is biharmonic.
\item[$(iii)$] $\h$ is unimodular.
\end{enumerate}

\end{proposition}

\begin{proof} In this case $\ker\xi=\h$, $\rho$ is the adjoint representation of $\h$,  $\om=0$ and from \eqref{pr5eq} we deduce that $\tau(\xi)=-U^\h$ and the equivalence of $(i)$ and $(iii)$ follows. Since $\om=0$, according to Theorem \ref{theo1}, $\pi$ is biharmonic iff $\ad_{U^\h}+\ad_{U^\h}^*=0$ this implies that 
$\tr(\ad_{U^\h})=\langle U^\h,U^\h\rangle_\h=0$ and the equivalence of $(i)$ and $(ii)$ follows.
\end{proof}

\begin{theorem}\label{theo2} Let $\phi:G\too H$ be a Riemannian submersion between two Riemannian Lie groups. Suppose that the metric on $H$ is flat and $\ker\xi$ is unimodular or the metric on $H$ is flat and $\om=0$. Then $\phi$ is biharmonic.

\end{theorem}

\begin{proof} One can deduce easily from \eqref{pr5eq} and \eqref{condition} that if $\ker\xi$ is unimodular or $\om=0$ then $\tau(\xi)\in[\h,\h]^\perp$. Now if the metric on $H$ is flat, it was shown in \cite{boucetta} that for any $u\in[\h,\h]^\perp$ $u^\ell$ is parallel and Theorem \ref{theo1} permits to conclude.
\end{proof}

We end this section by an important remark involving Riemannian submersion between Riemannian Lie groups. Let $\phi:G\too H$ and $\psi:H\too K$ two homomorphisms between Riemannian Lie groups. Suppose that $\phi$ is a Riemannian submersion and denote by $\xi$ and $\rho$ the differential at the neutral element of $\phi$ and $\psi$, respectively. We have
\begin{equation}\label{composition}
\tau(\rho\circ\xi)=\tau(\rho)+\rho(\tau(\xi)).
\end{equation}
This formula implies that if $\phi$ is harmonic then $\psi$ is biharmonic (resp. harmonic) iff $\psi\circ\phi$ is biharmonic (resp. harmonic).
 
 \section{When harmonicity and biharmonicity are equivalent}\label{section6}
 
 The following result is similar to Jiang's Theorem where compacity is replaced by unimodularity.
 \begin{theorem}\label{theo3} Let $\phi:G\too H$ be a homomorphism between two Riemannian Lie groups such that $R^H\leq0$ and $\G$ is unimodular. Then $\phi$ is harmonic iff it is biharmonic.
 
 \end{theorem}
 
 \begin{proof} Suppose that $\phi$ is biharmonic. Then according to \eqref{tension} we get
 \[ \sum_{i=1}^n\left(\langle B_{\xi(e_i)}\tau(\xi),B_{\xi(e_i)}\tau(\xi)\rangle_\h
 +\langle K^H(\tau(\xi),\xi(e_i))\xi(e_i),\tau(\xi)\rangle_\h
 \right)=0. \]
 Since the curvature is negative we deduce that $B_{\xi(e_i)}\tau(\xi)=0$ for any $i=1,\ldots,n$. Now since $\G$ is unimodular $U^\G=0$ and hence $\tau(\xi)=U^\xi$ so
 \[ \langle \tau(\xi),\tau(\xi) \rangle_\h
 =\sum_{i=1}^n\langle B_{\xi(e_i)}\xi(e_i),\tau(\xi) \rangle_\h
 =-\sum_{i=1}^n\langle \xi(e_i),B_{\xi(e_i)}\tau(\xi) \rangle_\h=0 \]
 and hence $\phi$ is harmonic.
 \end{proof}
 
 \begin{corollary}\label{cotheo3} Let $\phi:G\too H$ be a homomorphism between two Riemannian Lie groups such that $R^H\leq0$ and $Ric^G\geq0$.  Then $\phi$ is harmonic iff it is biharmonic.

 \end{corollary}
 
 \begin{proof} This is a consequence of Theorem \ref{theo3} and the fact that a Lie group which admits a left invariant Riemannian metric with non-negative Ricci curvature must be unimodular (see \cite{milnor}  Lemma 6.4).
 \end{proof}
 
 \begin{remark}Actually, this corollary follows from a general theorem (see \cite{seddik} Theorem 3.1). The Lie groups which admit left invariant Riemannian metrics with $R\leq0$ have been classified by Azencott and Wilson \cite{wilson} and are all solvable.
 
 \end{remark}
 Since $|\tau(\phi)|=cst$ for any homomorphism of Riemannian Lie groups, we get the following results proved in a more general sitting by Oniciuc  \cite{oniciuc} in Propositions 2.2, 2.4, 2.5, 4.3.
 
 \begin{theorem}\label{theo4}Let $\phi:G\too H$ be a homomorphism between two Riemannian Lie groups.
 In each for the following cases,  $\phi$ is biharmonic iff it is harmonic:
 \begin{enumerate}\item $R^H\leq0$ and $\phi$ is a Riemannian immersion.
 \item $Ric^H\leq0$, $\phi$ is a Riemannian immersion and $\dim H=\dim G+1$. 
 
 \item $R^H<0$ and $\mathrm{rank}\xi>1$.
 \item $Ric^H<0$ and $\phi$ is a Riemannian submersion. 
 
 \end{enumerate}\end{theorem}
 The following results are specific to our context.
 \begin{theorem}\label{theo5}Let $\phi:G\too H$ be a homomorphism between two Riemannian Lie groups.
 In the following cases the harmonicity of $\phi$ and its biharmonicity are equivalent:
 \begin{enumerate}\item $H$ is 2-step nilpotent and $\G$ is unimodular.
 \item $\phi$ is a Riemannian submersion and $\G$ is unimodular.
 \item $\phi$ is a Riemannian submersion,  $\ker\xi^\perp$ is a subalgebra of $\G$ and $\G$ is unimodular.
 \item $\phi$ is a Riemannian submersion,  $\ker\xi$ is unimodular, $\dim H=2$ and $H$ is non abelian.

 \end{enumerate}\end{theorem}

 \begin{proof}\begin{enumerate}\item Since $\G$ is unimodular then $\tau(\xi)=U^\xi$. Now since $\h$ is 2-step nilpotent then $[\h,\h]\subset Z(\h)$ and $\ad_u\circ\ad_v=0$ for any $u,v\in\h$. From Proposition \ref{pr2} we deduce that $U^\xi\in Z(\h)^\perp\subset[\h,\h]^\perp$ and hence
 \[ \tr(\xi^*\circ\ad_{U^\xi}^*\circ\ad_{U^\xi}\circ\xi)=0. \]
 This equivalent to $\ad_{U^\xi}\circ\xi=0$. So
 \[ \langle U^\xi,U^\xi\rangle_\h=\sum_{i=1}^n
 \langle B_{\xi(e_i)}\xi(e_i),U^\xi\rangle_\h=
 \sum_{i=1}^n
  \langle \xi(e_i),[U^\xi,\xi(e_i)]^\h\rangle_\h=0,
  \]and hence $\phi$ is harmonic.
  \item Suppose that $\phi$ is biharmonic. Since $\G$ is unimodular $\tau(\xi)=U^\xi$, $\ker\xi$ is unimodular and according to Theorem \ref{theo1}, $U^\xi$ is parallel an hence the Killing. But we have seen in Proposition \ref{pr2} that $U^\xi$ is orthogonal to the space $\mathrm{Kill}(h)$ and hence $U^\xi=0$.
  \item The same argument as above.
  
  \item According to Theorem \ref{theo1}, to prove this assertions it suffices to prove that for any left invariant metric on the 2-dimensional non abelian Lie group there is no non trivial parallel left invariant vector field. Suppose that $\dim\h=2$ non abelian. Then there exists an orthonormal basis $(e,f)$ such that $[e,f]=ae$. A direct computation gives 
  \[ B_ee=-af,\; B_ef=ae,\; B_fe=0,\; B_ff=0. \]
  $\al=\al_1e^*+\al_2f^*$ is parallel iff
  $ -a\al_1f+a\al_2e=0.$ So $\al=0$.
 \end{enumerate}
 
 \end{proof}
 
 \section{Some general methods for building examples}\label{section7}
 
 In this section, following our study in Sections \ref{section3} and \ref{section4} we give some general methods to builds large classes of harmonic and biharmonic homomorphisms. The following methods are based on Propositions \ref{submersion} and \ref{pr6}, \eqref{pr5eq} and Theorem \ref{theo1}.
 \subsection{How to build harmonic submersions between Riemannian Lie groups}
 \begin{enumerate}\item Choose two Lie algebras $\h$ and $\n$ with two Euclidean products  $\prs_1$ and $\prs_2$ on $\h$ and an Euclidean product $\prs_\n$ on $\n$.
 \item Compute $\tau(\mathrm{Id}_\h)$ where $\mathrm{Id}_\h:(\h,\prs_1)\too(\h,\prs_2)$.
 \item Construct $\rho:\h\too\mathrm{Der}(\n)$ and $\om\in\wedge^2\h^*\otimes\n$ satisfying \eqref{condition} and for any $h\in\h$, $\tr(\rho(h))=\langle h,\tau(\mathrm{Id}_\h)\rangle_1$. 
 \item The projection $(\n\oplus\h,[\;,\;],\prs_\n\oplus\prs_1)\too(\h,[\;,\;]^\h,\prs_2)$ is harmonic. The bracket $[\;,\;]$ is given by \eqref{bracket}.

 \end{enumerate}
 
 \subsection{How to build biharmonic submersions between Riemannian Lie groups}
  \begin{enumerate}\item Choose two Lie algebras $\h$ and $\n$ with two Euclidean products  $\prs_1$ and $\prs_2$ on $\h$ and an Euclidean product $\prs_\n$ on $\n$ such that $\mathrm{Id}_\h:(\h,\prs_1)\too(\h,\prs_2)$ is biharmonic.
  \item Construct $\rho:\h\too\mathrm{Der}(\n)$ and $\om\in\wedge^2\h^*\otimes\n$ satisfying \eqref{condition} and for any $u\in\h$, $\tr(\rho(h))=0$. 
  \item The projection $(\n\oplus\h,[\;,\;],\prs_\n\oplus\prs_1)\too(\h,[\;,\;]^\h,\prs_2)$ is biharmonic. The bracket $[\;,\;]$ is given by \eqref{bracket}.
  
  \end{enumerate}

 \subsection{How to build biharmonic Riemannian submersions between Riemannian Lie groups: first method}
   \begin{enumerate}\item Choose two Lie algebras $\h$ and $\n$ with two Euclidean products  $\prs_1$ on $\h$ and  $\prs_\n$ on $\n$.
   
   \item Construct $\rho:\h\too\mathrm{Der}(\n)$ a representation such that  $\tr\circ\rho$ is parallel, i.e., for any $u,v\in\h$, $\tr(\rho(A_uv))=0$, where $A$ is the Levi-Civita product on $\h$. 
   \item The projection $(\n\oplus\h,[\;,\;],\prs_\n\oplus\prs_1)\too(\h,[\;,\;]^\h,\prs_1)$ is biharmonic. The bracket $[\;,\;]$ is given by \eqref{bracket} with $\om=0$.
   
   \end{enumerate}

  \subsection{How to build biharmonic Riemannian submersions between Riemannian Lie groups: second method}
     \begin{enumerate}\item Choose two Lie algebras $\h$ and $\n$ with two Euclidean products  $\prs_1$ on $\h$ and  $\prs_\n$ on $\n$. Take $\n$ unimodular.
     
      \item Construct $\rho:\h\too\mathrm{Der}(\n)$ and $\om\in\wedge^2\h^*\otimes\n$ satisfying \eqref{condition} such that  $\tr\circ\rho$ is parallel, i.e., for any $u,v\in\h$, $\tr(\rho(A_uv))=0$, where $A$ is the Levi-Civita product on $\h$. 
          
     \item The projection $(\n\oplus\h,[\;,\;],\prs_\n\oplus\prs_1)\too(\h,[\;,\;]^\h,\prs_1)$ is biharmonic. The bracket $[\;,\;]$ is given by \eqref{bracket}. 
     
     \end{enumerate}
     
     \subsection{How to build biharmonic Riemannian submersions between Riemannian Lie groups: third method}

     \begin{enumerate}\item Choose two Lie algebras $\h$ and $\n$ with two Euclidean products  $\prs_1$ on $\h$ and  $\prs_\n$ on $\n$. Take $\h$ unimodular.
          
           \item Construct $\rho:\h\too\mathrm{Der}(\n)$ and $\om\in\wedge^2\h^*\otimes\n$ satisfying \eqref{condition} such that  $\tr\circ\rho$ is a Killing 1-form, i.e., for any $u,v\in\h$, $\tr(\rho(\ad_u^*v+\ad_v^*w))=0$. 
          \item The projection $(\n\oplus\h,[\;,\;],\prs_\n\oplus\prs_1)\too(\h,[\;,\;]^\h,\prs_1)$ is biharmonic. The bracket $[\;,\;]$ is given by \eqref{bracket}. 
          
          \end{enumerate}
          
On all the methods above, the crucial point is to solve \eqref{condition}, the following lemma gives an easy way of finding many solutions of these equations. 
\begin{lemma}\label{lef} Let $(\n,\h)$ a couple of Lie algebras such that $\mathrm{Der}(\n)=\ad(\n)$. Then $\rho:\h\too\mathrm{Der}(\n)$ and $\om\in\wedge^2\h^*\otimes\n$ satisfy \eqref{condition} iff there exists $F:\h\too\n$ a linear map and $\om_0\in\wedge^2\h^*\otimes Z(\n)$ such that
\[ \rho(u)=\ad_{F(u)},\om(u,v)=F([u,v])-[F(u),F(v)]+\om_0(u,v)\esp d\om_0=0. \]

\end{lemma}

\begin{proof} Since $\mathrm{Der}(\n)=\ad(\n)$ then $\rho(u)=\ad_{F(u)}$ and from \eqref{condition} we deduce that $\om$ must have the following form
\[ \om(u,v)=F([u,v])-[F(u),F(v)]+\om_0(u,v), \]where $\om_0$ takes its values in the center $Z(\n)$.
We have
\begin{eqnarray*}
\rho(u).\om(v,w)&=&[F(u),\om(v,w)]^{\n}=[F(u),F([v,w]^\h)]-[F(u),[F(v),F(w)]^{\n}]^{\n},\\
\om(u,[v,w]^\h)&=&F([u,[v,w]^\h]^\h)-[F(u),F([v,w]^\h)]^{\n}+\om_0(u,[v,w]^\h).
\end{eqnarray*}This shows that $d_\rho\om=0$ iff $d\om_0=0$.
\end{proof}

We end this paper by giving an example where we use Lemma \ref{lef} to illustrate the first method.
\begin{example}We take $\h$ the non abelian 2-dimensional Lie algebra endowed with two Euclidean products $\prs_1$ and $\prs_2$. There exists an $\prs_1$-orthonormal basis $(e_1,e_2)$ such that   $[e_1,e_2]=\al e_1$. We have
 \[ \tau(\mathrm{Id}_\h)=B_{e_1}e_1+B_{e_2}e_2+e_2. \]
 The condition $\tr(\rho(h))=\langle h,\tau(\mathrm{Id}_\h)\rangle_1$ is equivalent to
 \[ \tau(\mathrm{Id}_\h)=\tr(\rho(e_1))e_1+\tr(\rho(e_2))e_2. \]
 This is equivalent to the fact that $(\tr(\rho(e_1)),\tr(\rho(e_2)))$ is solution of the system
 \[ \left\{\begin{array}{lll} \langle e_1,e_1\rangle_2 x
 +\langle e_1,e_2\rangle_2 y&=&(\al+1)\langle e_1,e_2\rangle_2,\\
 \langle e_1,e_2\rangle_2 x
  +\langle e_2,e_2\rangle_2 y&=&-\al\langle e_1,e_1\rangle_2+\langle e_2,e_2\rangle_2.\end{array}
   \right.\leqno(S) \] 
   Let $\n$ be a non unimodular Euclidean Lie algebra and $F:\n\too\h$ an endomorphism such that $F(U^\n)=x_0e_1+y_0e_2$ where $(x_0,y_0)$ is the unique solution of $(S)$. Put $\rho(h)=\ad_{F^*(u)}$. For any $\om\in\wedge^2\h^*\otimes Z(\n)$, $(\rho,\om)$ satisfy \eqref{condition} where
   \[ \om(u,v)=F^*([u,v])-[F^*(u),F^*(v)]+\om_0(u,v), \] and $F^*:\h\too\n$ is the adjoint of $F$ with respect to $(\prs_1,\prs_\n)$.

\end{example}

 \section{Biharmonic homomorphisms between low dimensional Lie groups} \label{section8}
 We end this work by clarifying the situation in dimension 2. We denote by $E(1)$ the 2-dimensional Lie group of rigid motions of the real line and by $\G_2$ its Lie algebra.

  \begin{proposition}\label{pr2d} Let $g_1,g_2$ be two left invariant Riemannian metrics on $E(1)$. Then the following holds:
  \begin{enumerate}\item[$(i)$] A Riemannian immersion $i:\R\too (E(1),g_1)$ is minimal iff $d_0i(\R)$ is orthogonal to $[\G_2,\G_2]$.

  \item[$(ii)$] Any homomorphism $\chi:(E(1),g_1)\too\R$ is biharmonic never harmonic unless it is constant.
  \item[$(iii)$] If $\phi:(E(1),g_1)\too (E(1),g_2)$ is a non constant homomorphism which is harmonic then there exists a constant $\la>0$ such that $\phi^*g_2=\la g_1$. 
  \item[$(iv)$] If $\phi:(E(1),g_1)\too (E(1),g_2)$ is an homomorphism which is biharmonic non harmonic then there exists a Riemannian immersion $i:\R\too(E(1),g_2)$ and a biharmonic homomorphism $\chi:(E(1),g_1)\too\R$ such that $\phi=i\circ\chi$.
  
  \end{enumerate}
  
  \end{proposition}
  
  \begin{proof} Note first that $\G_2$ is not unimodular. Put $\prs_i=g_i(\mathrm{Id}_{\R})$ for $i=1,2$.
  \begin{enumerate}\item[$(i)$] Denote by $A$ the Levi-Civita product associated to $(\G_2,\prs_1)$. To show the assertion it suffices to show that for any $u\in\G_2\setminus\{0\}$, $A_uu=0$ iff $u\in[\G_2,\G_2]^\perp$. Indeed, if $v$ is such that $\{u,v\}$ is a basis of $\G_2$, we have
  \[ \langle A_uu,u\rangle_{1}=0\esp \langle A_uu,v\rangle_{1}=\langle u,[v,u]\rangle_{g_1}.  \]Since $[u,v]$ is a generator of $[\G_2,\G_2]$ we can conclude. 
  \item[$(ii)$]  We have shown in Proposition \ref{prab} that $\chi$ is biharmonic. If $\xi:\G_2\too\R$ denote the differential of $\chi$, we can see from the first relation in Proposition \ref{pr2} that $U^\xi=0$. So $\tau(\xi)=-\xi(U^{\G_2})$. Now $U^\G\in[\G_2,\G_2]^\perp$ and $[\G_2,\G_2]\subset\ker\xi$ so $\tau(\xi)=0$ iff $\xi=0$.
  
  \item[$(iii)$] Suppose that $\phi$ is harmonic and denote by $\xi:\G_2\too\G_2$ the differential of $\phi$. Then there exists an $\prs_{1}$-orthonormal basis $(e,f)$ such that $[e,f]=ae$. Denote by $A$ and $B$ the Levi-Civita product associated respectively to $g_1$ and $g_2$. One can see easily that $U^{g_1}=-af$. The harmonicity of $\phi$ is equivalent to
  \[ B_{\xi(e)}\xi(e)+B_{\xi(f)}\xi(f)=\xi(U^{g_1})=-a\xi(f).\eqno(h) \]
  Since $e$ is a generator of $[\G_2,\G_2]$ and $\xi$ is a Lie algebra homomorphism then  $\xi(e)=\al e$. Put $\xi(f)=pe+qf$. Since $\xi$ is a Lie algebra homomorphism then 
   \[ \xi([e,f])=[\xi(e),\xi(f)]=a\al e=[\al e,pe+qf]=\al aq e. \]
   If $\al=0$ then from $(h)$ we get
   $-a\xi(f)=B_{\xi(f)}\xi(f)$ and hence
   \[ - a\langle \xi(f),\xi(f)\rangle_{g_2}=-\langle B_{\xi(f)}\xi(f),\xi(f)\rangle_{g_2}\stackrel{\eqref{lc}}=0, \]and hence
    $\xi(f)=0$. If $\al\not=0$ then $q=1$ and from $(h)$ we get
   \[ -a\xi(f)=\al^2B_ee+B_{\xi(f)}\xi(f). \]
   So
   \[- a\langle \xi(f),\xi(f)\rangle_{g_2}=\al^2\langle [\xi(f),e],e\rangle_{g_2}=
   -a\al^2\langle e,e\rangle_{g_2}, \]
   and
   \[- a\langle \xi(f),e\rangle_{g_2}=\langle [e,\xi(f)],\xi(f)\rangle_{g_2}=
     a\langle e,\xi(f)\rangle_{g_2}. \]
    So
    \[ \langle \xi(f),\xi(f)\rangle_{g_2}=\langle \xi(e),\xi(e)\rangle_{g_2}=
    \al^2\langle e,e\rangle_{g_2}\esp \langle \xi(e),\xi(f)\rangle_{g_2}=0, \]  
   which completes the proof of the assertion.
   
   \item[$(iv)$]By multiplying $\prs_2$ by a positive constant if necessary, we can suppose that there exists a generator $e$ of $[\G_2,\G_2]$ such that $\langle e,e\rangle_1=\langle e,e\rangle_2=1$. Choose $f$ orthogonal to $e$ with respect to $\prs_1$, $\langle f,f\rangle_1=1$ and $f'$ orthogonal to $e$ with respect to $\prs_2$ with $\langle f',f'\rangle_2=1$.
   We have $[e,f]=ae$ and $[e,f']=be$ and
   \[ A_ee=-af,\; A_ef=ae,\; A_fe=0,\; A_ff=0 \]
   and
   \[ B_ee=-bf',\; B_ef'=be,\; B_{f'}e=0,\; B_{f'}f'=0 \]
   if $R$ is the curvature of $(\G_2,\prs_2)$ then
   \begin{eqnarray*}
   R(e,f')e&=&B_eB_{f'}e-B_{f'}B_ee-bB_ee=b^2f',\\
   R(e,f')f'&=&B_eB_{f'}f'-B_fB_ef'-bB_ef'=-b^2e.
   \end{eqnarray*}
   Now $\phi$ is biharmonic iff
   \[ B_{\xi(e)}B_{\xi(e)}\tau(\xi)+B_{\xi(f)}B_{\xi(f)}\tau(\xi)+
   R(\tau(\xi),\xi(e))\xi(e)+R(\tau(\xi),\xi(f))\xi(f)=B_{\xi(U^{g_1})}\tau(\xi) \]where
   \[ \tau(\xi)=B_{\xi(e)}\xi(e)+B_{\xi(f)}\xi(f)+a\xi(f). \]
   Put $\xi(e)=\al e$ and $\xi(f)=pe+qf'$. We have
   \[ \xi([e,f])=[\xi(e),\xi(f)]=a\al e=[\al e,pe+qf']=\al bq e. \]
   We have
   \begin{eqnarray*}
   \tau(\xi)&=&B_{\xi(e)}\xi(e)+B_{\xi(f)}\xi(f)+a\xi(f)\\
   &=&-\al^2bf'+p(-bpf'+qbe)+a(pe+qf')\\
   &=&p(a+qb)e+(aq-bp^2-\al^2b)f'.
   \end{eqnarray*}
   So
   \[ \tau(\xi)=p(a+qb)e+(aq-bp^2-\al^2b)f'=Qe+Pf. \]
   Moreover,
   \begin{eqnarray*}
   B_{\xi(e)}B_{\xi(e)}\tau(\xi)&=&\al^2B_e(-Qbf'+Pbe)=\al^2(-Qb^2e-Pb^2f')\\
   &=&-\al^2b^2\tau(\xi),\\
   B_{\xi(f)}B_{\xi(f)}\tau(\xi)&=&-p^2b^2\tau(\xi),\\
   R^2(\tau(\xi),\xi(e))\xi(e)&=&\al^2PR^2(f',e)e=-\al^2Pb^2f',\\
   R^2(\tau(\xi),\xi(f))\xi(f)&=&(Qq-Pp)R^2(e,f')\xi(f)=(Qq-Pp)(pb^2f'-qb^2e),\\
   -B_{\xi(U^{g_1})}\tau(\xi)&=&aB_{\xi(f)}\tau(\xi)=ap(-Qbf'+Pbe).
   \end{eqnarray*}
   We have
   \begin{eqnarray*}
   \langle R(\tau(\xi),\xi(e))\xi(e),\tau(\xi)\rangle_2&=&-\al^2P^2b^2,\\
   \langle R(\tau(\xi),\xi(f))\xi(f),\tau(\xi)\rangle_2&=&b^2(Qq-Pp)(Pp-Qq).
   \end{eqnarray*}
   So if $\tau(\xi)\not=0$ then $\al=p=0$ and hence $\xi(e)=0$ and $\xi(f)=qf'$.
 If we define $i_0:\R\too\G_2$ and $\xi_0:\G_2\too\R$ by $i_0(1)=qf'$, $\xi_0(e)=0$ and $\xi_0(f)=1$ then $\xi=i_0\circ\xi_0$ and we can integrate $i_0$ and $\xi_0$ to get the desired homomorphisms. \qedhere 
  \end{enumerate}
     \end{proof}

\bigskip
\address{ 
Faculty of Sciences and Technology \\
Cadi-Ayyad  University \\
BP 549 Marrakesh \\
Morocco
}
{m.boucetta@uca.ma}
\address{
Mathematical Institute \\
Saida University \\
BP 138 Saida \\
Algeria
}
{seddik.ouakkas@gmail.com}

\end{document}